\documentclass[11pt]{article}

\usepackage[a4paper,margin=27mm]{geometry}
\usepackage[T1]{fontenc}
\usepackage{lmodern}
\usepackage{microtype}
\usepackage{amsmath,amssymb,amsthm,mathtools,bm}
\usepackage{booktabs,array,tabularx}
\usepackage{enumitem}
\usepackage{graphicx}
\usepackage{xcolor}
\usepackage[round,authoryear]{natbib}

\usepackage{indentfirst}

\definecolor{linkblue}{RGB}{45,75,135}
\usepackage{hyperref}
\hypersetup{
    colorlinks=true,
    linkcolor=linkblue,
    filecolor=magenta,      
    urlcolor=blue, 
    citecolor=teal
}
\usepackage{cleveref}

\allowdisplaybreaks
\numberwithin{equation}{section}

\newtheorem{theorem}{Theorem}[section]
\newtheorem{proposition}[theorem]{Proposition}
\newtheorem{lemma}[theorem]{Lemma}
\newtheorem{corollary}[theorem]{Corollary}
\theoremstyle{definition}

\theoremstyle{remark}
\newtheorem{remark}[theorem]{Remark}

\DeclareMathOperator{\Res}{Res}

\newcommand{\E}{\mathbb E}
\newcommand{\R}{\mathbb R}
\newcommand{\Q}{\mathbb Q}
\newcommand{\Z}{\mathbb Z}

\newcommand{\law}{\mathcal L}
\newcommand{\dd}{\,\mathrm d}
\newcommand{\Ctwo}{\mathcal C_2}
\newcommand{\Dtwo}{D_2}
\newcommand{\vct}[1]{\bm{#1}}

\newcommand{\PaperAuthors}{%
    Alexandra-Ionela Andriciuc$^{1, 2}$\thanks{E-mails:
\texttt{andriciucalexandra@yahoo.com}; \texttt{alexandra-ionela.andriciuc@unibuc.ro}}, 
    Ionel Popescu$^{1, 2}$\thanks{E-mails:
\texttt{ionel.popescu@fmi.unibuc.ro};
\texttt{ionel.popescu@imar.ro}}, 
    David-Corneliu Turturean$^{3}$\thanks{E-mail:
\texttt{davidturturean@gmail.com}}%
}

\newcommand{\PaperAffiliations}{%
    \footnotesize
    $^{1}$ University of Bucharest, Faculty of Mathematics and Computer Science, Bucharest, Romania \\
    \footnotesize
    $^{2}$ Simion Stoilow Institute of Mathematics of the Romanian Academy, Bucharest, Romania \\
    \footnotesize
    $^{3}$ Massachusetts Institute of Technology, Cambridge, USA%
}

\title{\textbf{Non-Gaussianity of the Stagnation Law \\ in Particle Swarm Optimization}\\[0.35em]
\large An exact moment certificate}

\author{\PaperAuthors \\[0.5em] \PaperAffiliations} 
\date{}
\begin{document}
\maketitle

\begin{abstract}
We study one-dimensional particle swarm optimization during stagnation, with two fixed distinct attractors and equal independent uniform acceleration ranges.  The position then satisfies a second-order random affine recurrence.  For inertia $w$ and acceleration range $c$, we prove that throughout the open mean-square stability region
\[
 -1<w<1,\qquad c>0,\qquad 12(1-w^2)-c(7-5w)>0,
\]
no invariant position marginal, and hence no limiting position marginal, can be Gaussian. This solves the open Problem 18 in \cite{ParticleSwarmProblems}. The proof compares the stationary moment equations with the Gaussian moment identities through order eight.  A Hermite-polynomial formulation gives explicit fourth- and sixth-order compatibility conditions whose common solutions lie on a degree-$107$ polynomial branch.  Exact eighth-order equations exclude every point on that branch.  The final certificate is verified using arithmetic modulo $23$ and independently modulo $1{,}000{,}003$.  A separate raw-moment implementation produces exact polynomials $Q_4,Q_6,Q_8$ in $(w,c)$ and verifies the same obstruction over the rational numbers.  The fourth- and sixth-order curves have a genuine admissible intersection, but the eighth-order condition removes it, showing why low-order Gaussian diagnostics are insufficient.  All code, exact polynomials, logs, and plot-validation data are supplied as online resources.
\end{abstract}


\section{Introduction}

Particle swarm optimization (PSO) was introduced by Kennedy and Eberhart as a population-based stochastic search method \citep{KennedyEberhart1995,eberhart1995new}.  Its compact update rule produces a nontrivial random dynamical system.  Stability and convergence were studied from deterministic and stochastic viewpoints by, among others, \citet{ClercKennedy2002,Trelea,CampanaFasanoPinto2006,kadirkamanathan2006stability,Differentnoisedistributions,BonyadiMichalewicz2015,ErskineJoyceHerrmann2017} and \citet{cleghorn2018particle}.  The sampling distribution and its moments during stagnation were analyzed by \citet{Clerc2006,Poli2007Moments,PoliBroomhead2007,jiang2007stochastic,poli2008dynamics,Poli2009,liu2015order,cleghorn2019particle, dong2019order}.

Problem~18, ``Bell curve,'' on the Particle Swarm Central problem list asks for a proof that, in one dimension with fixed personal and neighborhood best positions, the empirical position histogram is Gaussian-like but not Gaussian \citep{ParticleSwarmProblems}.  We give a complete algebraic resolution of this problem.

The proof follows a concrete moment strategy.  If the stationary position were Gaussian, its centered moments would obey
\[
 \E X^4=3(\E X^2)^2,\qquad
 \E X^6=15(\E X^2)^3,\qquad
 \E X^8=105(\E X^2)^4.
\]
The PSO recursion supplies a second set of equations for moments of each position and for the mixed moments of two consecutive positions.  At each total degree these equations are linear in the new mixed moments.  Eliminating those unknowns produces polynomial conditions on the algorithmic parameters.

Two versions of this calculation are developed.
\begin{enumerate}[leftmargin=1.7em]
\item In the \emph{Hermite formulation}, Gaussian one-dimensional moments disappear automatically from the equations.  The fourth- and sixth-order conditions reduce the parameter set to the roots of one degree-$107$ polynomial.  The eighth-order equations then give an exact contradiction.
\item In the \emph{raw-moment formulation}, one solves an $(n+1)\times(n+1)$ system at each even degree $n$ and forms the usual Gaussian moment defect.  This produces the explicit curves $Q_4=0$, $Q_6=0$, and $Q_8=0$ shown in Figures \ref{fig:curves} and \ref{fig:zoom}, thus supplying an independent exact audit.
\end{enumerate}
The two calculations are tied by an exact coordinate identity and produce the same degree-$107$ elimination polynomial.  This makes the final conclusion independent of a single choice of basis or software path.

The computer-assisted part is deliberately narrow.  The probabilistic reductions, stability region, moment equations, and logical implication of the certificate are given in the paper.  The computer performs exact polynomial arithmetic and exact linear algebra.  Floating-point calculations are used only to draw and diagnose the parameter curves; the computational supplement quantifies their resolution and makes clear that they are not part of the proof.

\paragraph{Scope.}
The theorem concerns one spatial dimension, two fixed distinct attractors, independent acceleration variables with the common law $\operatorname{Unif}[0,c]$, and parameters in the open mean-square region.  It does not by itself cover unequal cognitive and social ranges, moving attractors, velocity clamping, or coupled multidimensional variants.

\paragraph{Organization.}
Sections~\ref{sec:model}--\ref{sec:second} formulate the model and derive the admissible parameter domain.  Sections~\ref{sec:hermite}--\ref{sec:eighth} give the Gaussian-moment contradiction.  Section~\ref{sec:raw-bridge} presents the independent raw-moment formulation, \Cref{sec:geometry} explains the parameter curves, and \Cref{sec:audit} summarizes verification and reproducibility.

\section{Original PSO model, normalization, and main theorem}\label{sec:model}

Let $P,G\in\R$ be fixed attractors and let
\[
 \phi_{1,t},\phi_{2,t}\stackrel{\mathrm{iid}}{\sim}\operatorname{Unif}[0,c],
 \qquad t\ge1,
\]
with all variables independent across indices and time.  Eliminating velocity from the usual one-dimensional PSO update gives
\begin{equation}\label{eq:original-scalar}
 \mathsf X_{t+1}
 =\bigl(1+w-\phi_{1,t+1}-\phi_{2,t+1}\bigr)\mathsf X_t
 -w\mathsf X_{t-1}
 +P\phi_{1,t+1}+G\phi_{2,t+1}.
\end{equation}
Equivalently, for $Z_t=(\mathsf X_t,\mathsf X_{t-1})^{\mathsf T}$,
\begin{equation}\label{eq:matrix-model}
 Z_{t+1}=A_{t+1}Z_t+B_{t+1},
 \quad
 A_t=\begin{pmatrix}
 1+w-\phi_{1,t}-\phi_{2,t}&-w\\[1mm]1&0
 \end{pmatrix},
 \quad
 B_t=\binom{P\phi_{1,t}+G\phi_{2,t}}{0}.
\end{equation}
Define
\begin{equation}\label{eq:D2}
 \Dtwo(w,c):=12(1-w^2)-c(7-5w)
\end{equation}
and
\begin{equation}\label{eq:C2}
 \Ctwo:=\{(w,c)\in\R^2:-1<w<1,\ c>0,\ \Dtwo(w,c)>0\}.
\end{equation}
This is the standard open second-moment stability region for the stagnation model.  Indeed, if $a=1+w-\phi_1-\phi_2$, then the homogeneous covariance operator on the basis $(x^2,xy,y^2)$ is
\begin{equation}\label{eq:T2}
 T_2=\begin{pmatrix}
 \E[a^2]&-2w\E[a]&w^2\\
 \E[a]&-w&0\\
 1&0&0
 \end{pmatrix},
\end{equation}
where
\[
 \E[a]=1+w-c,
 \qquad
 \E[a^2]=(1+w)^2-2c(1+w)+\frac{7c^2}{6}.
\]
\begin{lemma}[Second-moment stability region]\label{lem:stability-region}
For $-1<w<1$ and $c>0$, one has $\rho(T_2)<1$ if and only if $\Dtwo(w,c)>0$.  In particular,
\begin{equation}\label{eq:detI-T2}
 \det(I-T_2)=\frac{c\Dtwo(w,c)}6.
\end{equation}
\end{lemma}

\begin{proof}
The characteristic polynomial is $\lambda^3+a_1\lambda^2+a_2\lambda+a_3$, where
\begin{align*}
 a_1&=-\frac{7c^2}{6}+2c(1+w)-w^2-w-1,\\
 a_2&=\frac{5c^2w}{6}-2cw(1+w)+w^3+w^2+w,
 \qquad a_3=-w^3.
\end{align*}
For a real monic cubic, the Schur--Jury inequalities (\cite{jury1962simplified,choo2011elementary}) are:
\[
 |a_3|<1,\quad
 1+a_1+a_2+a_3>0,\quad
 1-a_1+a_2-a_3>0,\quad
 1-a_2+a_1a_3-a_3^2>0.
\]
The first is exactly $|w|<1$, and the second is $c\Dtwo/6>0$.
For the third inequality, multiplication by $6$ gives
\[
 (5w+7)c^2-12(1+w)^2c+12(1+w)(w^2+1),
\]
a quadratic with positive leading coefficient and discriminant
$-96(w-1)(w+1)(w^2-2)<0$.

For the final inequality, let $P_w(c)$ denote six times its left side.  Then
\begin{align*}
 P_w(c)={}&w(7w^2-5)c^2+12w(1-w)(1+w)^2c\\
 &+6(1-w)^3(1+w)(w^2+w+1).
\end{align*}
At $c=0$ it is positive, and at
$c_{\max}=12(1-w^2)/(7-5w)$ it equals
\[
 \frac{-6(w-1)^2(w+1)(w^2+w+1)}{(5w-7)^2}
 \bigl(25w^3-143w^2+71w-49\bigr)>0.
\]
The cubic in parentheses is negative on $(-1,1)$: this is immediate for $w<0$, while for $0\le w<1$ it is at most $-143w^2+96w-49<0$.
If the leading coefficient of $P_w$ is nonpositive, endpoint positivity suffices by concavity.  If it is positive and $w>0$, then $P_w'(0)>0$, so $P_w$ is increasing.  Finally suppose it is positive and $w<0$.  If $9w^2+2w-3\ge0$, then $P_w'(c_{\max})\le0$, so convexity makes $P_w$ decreasing on the interval.  Otherwise $w>(-1-2\sqrt7)/9>-7/10$, and the discriminant
\[
 24w(w-1)^2(w+1)(w^2+1)(7w^3+6w^2+6w+5)
\]
is negative: the last factor is increasing and exceeds its value $1339/1000$ at $-7/10$.  Hence $P_w(c)>0$ throughout $0\le c\le c_{\max}$.  The four Jury inequalities are therefore equivalent to \Cref{eq:C2}.
\end{proof}

The usual affine-recursion construction therefore yields the unique square-integrable stationary law.  The non-Gaussian argument below is stronger than this existence statement: it rules out a Gaussian coordinate for \emph{any} invariant coupling and applies whenever the one-dimensional marginals are known to converge.

If $P=G$, centering at the common attractor removes the additive forcing.  Mean-square stability then leaves only the point mass at $P$ as a square-integrable stationary marginal.  Thus $P\ne G$ isolates the genuinely random, nondegenerate case.

From now on we assume that $P\ne G$.  Set
\begin{equation}\label{eq:affine-normalization}
 \mu:=\frac{P+G}{2},
 \qquad d:=\frac{G-P}{2},
 \qquad X_t:=\frac{\mathsf X_t-\mu}{d}.
\end{equation}
Then, a direct substitution reduce \Cref{eq:original-scalar} to
\begin{equation}\label{eq:normalized-recurrence}
 X_{t+1}
 =\bigl(1+w-\phi_{1,t+1}-\phi_{2,t+1}\bigr)X_t
 -wX_{t-1}-\phi_{1,t+1}+\phi_{2,t+1}, 
\end{equation}
with the normalized attractors being $-1$ and $1$.

\medskip

The following theorem is the main result of this paper. 

\begin{theorem}[Non-Gaussian stationary and limiting marginal]\label{thm:main}
Let $(w,c)\in\Ctwo$ and $P\ne G$.
 No invariant probability law for the two-coordinate Markov chain associated with \Cref{eq:matrix-model} can have a Gaussian first-coordinate marginal.

\end{theorem}

\section{Stationary coupling, centering, and symmetry}\label{sec:stationary}

For the normalized recurrence, write
\begin{equation}\label{eq:a-delta}
 A=1+w-U-V,
 \qquad D=V-U,
\end{equation}
where $U,V$ are independent $\operatorname{Unif}[0,c]$ variables.  The state transition is
\begin{equation}\label{eq:state-transition}
 (x,y)\longmapsto(Ax-wy+D,x),
\end{equation}
with fresh noise at each step.

Let $(X,Y)$ have an invariant law and let $(U,V)$ be fresh noise independent of $(X,Y)$.  Put
\begin{equation}\label{eq:Xprime}
 X'=AX-wY+D.
\end{equation}
Then
\begin{equation}\label{eq:pair-stationarity}
 (X',X)\stackrel d=(X,Y).
\end{equation}
Assume for contradiction that the common coordinate marginal is Gaussian.  If $m=\E X=\E Y$, then
\[
 m=(1+w-c)m-wm=(1-c)m,
\]
so $m=0$ because $c>0$.

The transition kernel commutes with simultaneous reflection.  More precisely, swapping $U$ and $V$ leaves $A$ unchanged and changes $D$ to $-D$, so the reflected image of an invariant law is invariant.  Averaging an invariant coupling with its reflection preserves its centered Gaussian marginals.  Hence we may assume
\begin{equation}\label{eq:central-symmetry}
 (X,Y)\stackrel d=(-X,-Y).
\end{equation}
If both marginals are Gaussian,  H\"older's inequality proves that every mixed moment of $(X,Y)$ is finite.  \Cref{eq:central-symmetry} then makes all mixed moments of odd total degree vanish.

\subsection{Non-Gaussianity of the $2-$dimensional process}

We assume $\{X_t\}_{t \ge 0}$ has reached its unique, zero-mean, symmetric invariant stationary distribution $\mu$ with marginal variance $\text{Var}(X_t) = \sigma_X^2 < \infty$, and lag-1 autocorrelation $\rho_1 = \text{Corr}(X_t, X_{t+1}) \in (-1, 1)$. 

Suppose also that $\{\phi_{1,t}\}_{t \ge 0}, \{\phi_{2,t}\}_{t \ge 0}$ from \Cref{eq:normalized-recurrence} are mutually independent sequences of i.i.d. random variables with:
\begin{equation}
\mathbb{E}[\phi_{1,t}] = \mathbb{E}[\phi_{2,t}] = \mu_\phi > 0, \quad \text{Var}(\phi_{1,t}) = \text{Var}(\phi_{2,t}) = \sigma_\phi^2 < \infty.
\end{equation}
In particular, for uniform noise $\phi_1, \phi_2 \sim U(0,c)$, the noise parameters are:
\begin{equation}
\mu_\phi = \frac{c}{2}, \quad \sigma_\phi^2 = \frac{c^2}{12}. \label{eq:uniform_moments}
\end{equation}

\begin{proposition}[Universal Quadratic Heteroscedasticity]\label{prop:calcul var cond}
For any noise distribution with variance $\sigma_\phi^2$, the conditional variance $\text{Var}(X_{t+1} \mid X_t = x)$ is strictly quadratic in $x$:
\begin{equation}
\text{Var}(X_{t+1} \mid X_t = x) = 2\sigma_\phi^2 (1 + x^2) + w^2 \text{Var}(X_{t-1} \mid X_t = x). \label{eq:var_quadratic_gen}
\end{equation}
\end{proposition}

\begin{proof}
By the Law of Total Variance conditioned on $X_t = x$:
\[
\text{Var}(X_{t+1} \mid X_t = x) = \mathbb{E}\left[\text{Var}(X_{t+1} \mid X_t, X_{t-1}) \mid X_t = x\right] + \text{Var}\left(\mathbb{E}[X_{t+1} \mid X_t, X_{t-1}] \mid X_t = x\right).
\]
Conditional on $(X_t=x, X_{t-1})$, $X_{t+1}$ varies solely through $\xi_{t+1}(x) := \phi_{2,t+1}(1 - X_t) - \Phi_{1,t+1}(1 + X_t). \label{eq:noise_def_gen}$. Since $\phi_{1,t}$ and $\phi_{2,t}$ are independent of $X_{t-1}$ with variance $\sigma_\phi^2$:
\[
\text{Var}(\xi_t(x) \mid X_t=x, X_{t-1}) = (1-x)^2 \sigma_\phi^2 + (1+x)^2 \sigma_\phi^2 = 2\sigma_\phi^2(1 + x^2).
\]
Because $\mathbb{E}[X_{t+1} \mid X_t=x, X_{t-1}] = (1+w-2\mu_\phi)x - w X_{t-1}$, taking the variance over $X_{t-1}$ given $X_t=x$ yields $w^2 \text{Var}(X_{t-1} \mid X_t = x)$. Adding both terms gives \eqref{eq:var_quadratic_gen}.
\end{proof}

\begin{theorem}[Universal Non-Gaussianity]
The joint stationary position distribution $(X_{t+1}, X_t)$ in stagnant PSO is non-Gaussian for ANY stochastic noise process with non-zero variance $\sigma_\phi^2 > 0$.
\end{theorem}

\begin{proof}
A necessary property of any Bivariate Normal distribution $(X, Y)$ is homoscedasticity of conditional variances, i.e., $\text{Var}(Y \mid X = x) = \sigma_Y^2 (1 - \rho^2) = \text{constant}$ for all $x$. 
From Proposition \ref{prop:calcul var cond}, the conditional variance $v_1(x) = \text{Var}(X_{t+1} \mid X_t = x)$ includes the explicit state-dependent term $2\sigma_\phi^2(1 + x^2)$. Since $\text{Var}(X_{t-1} \mid X_t = x) \ge 0$, we have:
\[
\text{Var}(X_{t+1} \mid X_t = x) \ge 2\sigma_\phi^2 (1 + x^2),
\]
which is strictly quadratic and non-constant in $x$ for all $\sigma_\phi^2 > 0$. Thus, $(X_{t+1}, X_t)$ violates conditional homoscedasticity and cannot be Bivariate Normal.
\end{proof}

\begin{remark}[Scope of Distributional Elimination via Heteroscedasticity]
We emphasize that conditional homoscedasticity ($\text{Var}(Y \mid X = x) = \text{constant}$) is a necessary property not only of the Bivariate Normal distribution, but of a broad class of bivariate models. Consequently, quadratic heteroscedasticity serves as a restrictive structural signature that narrows the valid candidate limiting distributions for PSO.
\end{remark}

\section{Second moments and the effective correlation}\label{sec:second}

Set
\[
 q=\E X^2=\E Y^2,
 \qquad r=\E[XY],
 \qquad \alpha=\E A=1+w-c.
\]
The elementary noise moments are
\begin{equation}\label{eq:noise-second}
 \E A^2=\alpha^2+\frac{c^2}{6},
 \qquad
 \E D^2=\frac{c^2}{6},
 \qquad
 \E[AD]=0.
\end{equation}
Taking the product of \Cref{eq:Xprime} with $X$ and then squaring \Cref{eq:Xprime} gives
\begin{align}
 r&=\alpha q-wr,\label{eq:cov-eq}\\
 q&=\left(\alpha^2+\frac{c^2}{6}+w^2\right)q
      -2w\alpha r+\frac{c^2}{6}.\label{eq:var-eq}
\end{align}
The case $q=0$ contradicts \Cref{eq:var-eq}, so $q>0$.  Define
\begin{equation}\label{eq:rho-def}
 \rho:=\frac rq.
\end{equation}
Solving \Cref{eq:cov-eq,eq:var-eq} yields
\begin{equation}\label{eq:rho-q}
 \rho=1-\frac{c}{1+w},
 \qquad
 q=\frac{c(1+w)}{\Dtwo(w,c)}.
\end{equation}
Thus $c=(1+w)(1-\rho)$, and \Cref{eq:C2} is equivalent to
\begin{equation}\label{eq:rho-domain}
 -1<w<1,
 \qquad
 \frac{7w-5}{7-5w}<\rho<1.
\end{equation}
If
\begin{equation}\label{eq:Qrho}
 \mathcal Q(w,\rho):=5-7w+(7-5w)\rho,
\end{equation}
then
\begin{equation}\label{eq:q-rho}
 q=\frac{(1+w)(1-\rho)}{\mathcal Q(w,\rho)},
\end{equation}
and $\mathcal Q>0$ on \Cref{eq:rho-domain}.

\section{The stationary Hermite hierarchy}\label{sec:hermite}

Normalize the stationary pair by
\begin{equation}\label{eq:ZW}
 Z=\frac{X}{\sqrt q},
 \qquad W=\frac{Y}{\sqrt q}.
\end{equation}
Both marginals are standard normal and $\E[ZW]=\rho$.  Let $H_n$ denote the probabilists' Hermite polynomial,
\[
 H_n(x)=(-1)^ne^{x^2/2}\frac{\dd^n}{\dd x^n}e^{-x^2/2},
\]
and define
\begin{equation}\label{eq:Mab}
 M_{ab}:=\E[H_a(Z)H_b(W)].
\end{equation}
Then
\begin{equation}\label{eq:Hermite-marginals}
 M_{a0}=M_{0a}=0\quad(a\ge1),
 \qquad M_{11}=\rho,
 \qquad M_{ab}=0\quad(a+b\text{ odd}).
\end{equation}
For raw normalized moments $R_{ab}=\E[Z^aW^b]$, the identity
\begin{equation}\label{eq:raw-Hermite}
 x^m=\sum_{j=0}^{\lfloor m/2\rfloor}
 \frac{m!}{2^j j!(m-2j)!}H_{m-2j}(x)
\end{equation}
expresses each $R_{ab}$ linearly in the $M_{ij}$.

Write
\begin{equation}\label{eq:nuij}
 \nu_{ij}:=\E\!\left[A^i\left(\frac D{\sqrt q}\right)^j\right]
 =\frac1{c^2q^{j/2}}\int_0^c\!\int_0^c
 (1+w-u-v)^i(v-u)^j\,\dd u\,\dd v.
\end{equation}
Interchanging $u$ and $v$ gives $\nu_{ij}=0$ for odd $j$.  The normalized update is
\begin{equation}\label{eq:Zprime}
 Z'=AZ-wW+\frac D{\sqrt q},
\end{equation}
and stationarity says $(Z',Z)\stackrel d=(Z,W)$.  Therefore, for every even $n$ and $0\le k<n$,
\begin{equation}\label{eq:master-Hermite}
 \sum_{i+j+\ell=n-k}
 \binom{n-k}{i,j,\ell}(-w)^j\nu_{i\ell}R_{i+k,j}
 -R_{n-k,k}=0.
\end{equation}
At fixed total degree $n$, these equations are linear in the cross-Hermite moments $M_{a,n-a}$, $1\le a\le n-1$.  Equation \Cref{eq:master-Hermite} is the sole algebraic input to the primary certificate.

\section{Fourth-order compatibility}\label{sec:fourth}

Put
\[
 u=M_{31},\qquad v=M_{22}.
\]
The equation with $(n,k)=(4,3)$ gives $M_{13}=-wu$.  The equations with $k=1,2$ determine
\begin{align}
 u&=\frac{2\rho(\rho-1)^2(1+w)(w^2-w+1)}{\mathcal H(w,\rho)},\label{eq:u-sol}\\
 v&=2\rho^2+\frac{(1+w)(1-\rho)^2}{3(1-w)}
 -\frac{2w\rho}{1-w}u,\label{eq:v-sol}
\end{align}
where
\begin{align}
\mathcal H(w,\rho)={}&5\rho^2w^3+7\rho^2w+2\rho w^3-2\rho w\notag\\
&+2w^4-5w^3+4w^2-3w+2.\label{eq:Hpoly}
\end{align}

\begin{lemma}\label{lem:H-positive}
The polynomial $\mathcal H$ is strictly positive throughout \Cref{eq:rho-domain}.
\end{lemma}

\begin{proof}
Set $\kappa=1-\rho$.  Then
\[
 0<\kappa<\frac{12(1-w)}{7-5w}
\]
and
\begin{equation}\label{eq:H-kappa}
 \mathcal H
 =w(5w^2+7)\kappa^2-12w(w^2+1)\kappa
 +2(w^2+1)(w^2+w+1).
\end{equation}
For $0<w<1$, this is a quadratic with positive leading coefficient and discriminant
\[
 -8w(w-1)(w^2+1)(5w^3-8w^2+4w-7)<0.
\]
Indeed, $5w^3-8w^2+4w-7\le-3w^2+4w-7<0$.  At $w=0$, $\mathcal H=2$.

For $-1<w<0$, \Cref{eq:H-kappa} is concave in $\kappa$, so its minimum on the closed endpoint interval occurs at an endpoint.  At $\kappa=0$ it equals $2(w^2+1)(w^2+w+1)>0$.  At the other endpoint it equals
\[
 \frac{2N(w)}{(5w-7)^2},
\]
where
\[
 N(w)=25w^6-45w^5+173w^4-66w^3-91w^2-21w+49.
\]
Writing $w=-x$, $0<x<1$, gives
\[
 N(-x)=25x^6+45x^5+66x^3+21x+(173x^4-91x^2+49)>0,
\]
because the final quadratic in $x^2$ has negative discriminant.
\end{proof}

The unused $k=0$ equation becomes
\begin{equation}\label{eq:fourth-residual}
 \frac{(\rho-1)^2(1+w)^2}{15\mathcal H(w,\rho)}F(w,\rho)=0,
\end{equation}
where
\begin{align}
F(w,\rho)={}&
\rho^4(5w^5+10w^4-109w^3-12w^2+30w)\notag\\
&+\rho^3(-164w^5+128w^4-224w^3+240w^2-108w)\notag\\
&+\rho^2(26w^6-30w^5+240w^4-48w^3+154w^2-138w+60)\notag\\
&+\rho(-28w^6+148w^5-256w^4+244w^3-284w^2+144w-48)\notag\\
&-10w^6+53w^5-134w^4+161w^3-110w^2+84w-24.
\label{eq:Fpoly}
\end{align}
All prefactors in \Cref{eq:fourth-residual} are nonzero in the admissible region, so Gaussianity forces
\begin{equation}\label{eq:Fzero}
 F(w,\rho)=0.
\end{equation}

\section{Sixth-order compatibility and elimination of the correlation}\label{sec:sixth}

The calculations below use a standard elimination device called a \emph{resultant}.  For two polynomials in the variable $\rho$, with coefficients depending on $w$, the resultant is a polynomial in $w$ alone.  Whenever the two original polynomials have a common value of $\rho$, their resultant must vanish.  We need only this one-way implication.  The computation is exact and amounts to determinant and polynomial-arithmetic operations on integer coefficients.

Introduce
\begin{equation}\label{eq:p6vector}
 \vct p_6=(M_{51},M_{42},M_{33},M_{24},M_{15})^{\mathsf T}.
\end{equation}
After substituting the fourth-order solution, the six equations \Cref{eq:master-Hermite} with $n=6$ form a linear system
\begin{equation}\label{eq:A6-system}
 A_6(w,\rho)\vct p_6=b_6(w,\rho).
\end{equation}
Multiplying rows $k=0,1,\ldots,5$, respectively, by
\begin{equation}\label{eq:scale6}
 42\mathcal H,\quad3\mathcal H,\quad15\mathcal H,
 \quad2\mathcal H,\quad\mathcal H,\quad1
\end{equation}
produces a polynomial system $\widehat A_6\vct p_6=\widehat b_6$ over $\Z[w,\rho]$.  Define $G\in\Z[w,\rho]$ by
\begin{equation}\label{eq:Gdef}
 \det[\widehat A_6\mid\widehat b_6]
 =3(\rho-1)^2(1+w)^5\mathcal H^4G(w,\rho).
\end{equation}
The generated polynomial has
\begin{equation}\label{eq:Gmetadata}
 \deg_wG=20,\qquad \deg_\rho G=12,
 \qquad \#\operatorname{supp}(G)=241.
\end{equation}
A solvable system with five unknowns has zero $6\times6$ augmented determinant; hence Gaussianity forces
\begin{equation}\label{eq:Gzero}
 G(w,\rho)=0.
\end{equation}

Eliminating $\rho$ gives the exact factorization
\begin{align}
\Res_\rho(F,G)
={}&1375941427200\,w^6(w-1)^5(w+1)^4(w^2+1)^3\notag\\
&\times(w^2-w+1)^2(w^2+w+1)^4
(2w^3-3w^2+w-2)^2R(w),
\label{eq:resultant-factorization}
\end{align}
where $R\in\Z[w]$ is primitive, has positive leading coefficient, and has degree $107$.

None of the displayed factors except $w^6$ can vanish for $-1<w<1$.  The factors $w\pm1$ are excluded, the three quadratic factors have no real roots, and
\[
 2w^3-3w^2+w-2<0\qquad(-1<w<1).
\]
For $w<0$ every term is negative; for $0\le w<1$ the expression is at most $-w^2+w-2<0$.  The case $w=0$ is impossible as well, because
\begin{align*}
 F(0,\rho)&=12(5\rho^2-4\rho-2),\\
 G(0,\rho)&=-720(113\rho^4-103\rho^3+62\rho^2-63\rho-30),
\end{align*}
and the resultant of the parenthesized polynomials is $68338\ne0$.  Therefore every admissible common zero of $F$ and $G$ satisfies
\begin{equation}\label{eq:Rzero}
 R(w)=0.
\end{equation}

Continuing the same elimination calculation produces a short remainder sequence with degrees
\begin{equation}\label{eq:subres-degrees}
 12,4,3,2,1,0.
\end{equation}
Its degree-one member has the form
\begin{equation}\label{eq:S1}
 S_1(w,\rho)=A_*(w)\rho+B_*(w).
\end{equation}
Exact arithmetic gives
\begin{equation}\label{eq:gcd-RA}
 \gcd_{\Q[w]}(R,A_*)=1.
\end{equation}
Every common zero of $F$ and $G$ satisfies this degree-one relation.  Consequently any admissible candidate satisfies
\begin{equation}\label{eq:rho-branch}
 \rho=-\frac{B_*(w)}{A_*(w)}.
\end{equation}
The phrase ``degree-$107$ branch'' will refer to \Cref{eq:Rzero,eq:rho-branch}; more precisely, $R$ is the univariate eliminant and the degree-one relation recovers $\rho$.

\section{The eighth-order contradiction}\label{sec:eighth}

Let
\begin{equation}\label{eq:h8vector}
 \vct h_8=(M_{71},M_{62},M_{53},M_{44},M_{35},M_{26},M_{17})^{\mathsf T}.
\end{equation}
The eight equations \Cref{eq:master-Hermite} with $n=8$ have the form
\begin{equation}\label{eq:A8-system}
 A_8(w,\rho)\vct h_8=b_8(w,\rho,\vct p_6).
\end{equation}
After multiplying the eight rows by
\begin{equation}\label{eq:scale8}
 15\mathcal H,\quad12\mathcal H,\quad42\mathcal H,
 \quad3\mathcal H,\quad5\mathcal H,
 \quad2\mathcal H,\quad\mathcal H,\quad1,
\end{equation}
all coefficients are integer polynomials in $(w,\rho)$; the right-hand side remains affine in the five sixth-order unknowns.

\subsection{Why reduce the equations modulo a prime?}

At this stage the formulas are exact but large.  We therefore perform the final contradiction using ordinary integer arithmetic with all numbers replaced by their remainders modulo a prime $p$.  This is not an approximation.  Addition, multiplication, division by a nonzero number, and Gaussian elimination are exact in the finite arithmetic system.

The logical point is elementary.  Suppose an integer polynomial identity held at an algebraic number in the original problem.  The minimal polynomial of that number would divide the polynomial expressing the identity.  Reducing the coefficients modulo any prime that does not destroy the leading terms preserves that divisibility.  Therefore an identity that is required over the real numbers must still be visible after reduction modulo every suitable prime.  Finding one suitable prime for which the required identity fails is enough to rule it out over the real numbers.

\begin{lemma}[Exact reduction test]\label{lem:lifting}
Let $\alpha$ be algebraic over $\Q$, let $T\in\Z[x]$ be its primitive minimal polynomial, and let $N\in\Z[x]$.  If $N(\alpha)=0$, then $T$ divides $N$ over $\Z[x]$.  If a prime $p$ does not divide the leading coefficient of $T$, each irreducible factor of $T$ after coefficient reduction modulo $p$ also divides the reduction of $N$.
\end{lemma}

\begin{proof}
Minimality gives divisibility over $\Q[x]$, and Gauss's lemma upgrades it to divisibility in $\Z[x]$ because $T$ is primitive.  Reducing the resulting product identity modulo $p$ preserves degrees when the leading coefficient of $T$ remains nonzero.  Every factor of the reduced $T$ must therefore divide the reduced $N$.
\end{proof}

A prime is called \emph{suitable for the certificate} when the degree-$107$ polynomial keeps its degree, has no repeated factor after reduction, the coefficients that must be divided by remain nonzero, the relevant linear systems remain nonsingular, and the final eighth-order residual can be tested on every factor branch.

\subsection{The primes 23 and 1,000,003}

The primary certificate uses $p=23$.  Reduction of $R$ modulo $23$ is square-free, its leading coefficient is $7$, and it factors into irreducible pieces of degrees
\begin{equation}\label{eq:p23-degrees}
 1,\ 2,\ 3,\ 6,\ 10,\ 30,\ 55.
\end{equation}
For each factor, the program represents $w$ by the residue class of the indeterminate, obtains $\rho$ from \Cref{eq:rho-branch}, solves the first five sixth-order equations, checks the sixth, solves the first seven eighth-order equations, and checks the eighth.  Every sixth-order residual is zero, while every eighth-order residual is nonzero.

As an independent modular cross-check of the same exact calculation, the certificate is rerun with $p=1{,}000{,}003$.  This prime gives factor degrees
\begin{equation}\label{eq:p1000003-degrees}
 1,\ 3,\ 9,\ 11,\ 25,\ 58
\end{equation}
and the same pass/fail pattern.  The larger prime has no special mathematical status.  It was retained because it changes the modular factorization substantially and therefore provides a useful cross-check.  The small prime makes the arithmetic easier to inspect; the large prime makes it less plausible that an unnoticed prime-specific degeneracy survived.  Either successful run is sufficient for the proof.

\begin{proposition}[Exact modular certificate]\label{prop:finite-certificate}
For both $p=23$ and $p=1{,}000{,}003$ the following statements hold after reduction modulo $p$.
\begin{enumerate}[label=\textup{(\roman*)}]
\item The degree-$107$ polynomial $R$ retains degree $107$ and is square-free.
\item On every irreducible factor branch of $R$, $A_*$ is nonzero, so \Cref{eq:rho-branch} is defined and both $F=0$ and $G=0$ hold.
\item The first five sixth-order rows form an invertible system; their solution satisfies the sixth row.
\item The first seven eighth-order rows form an invertible system; their solution fails the eighth row.
\end{enumerate}
\end{proposition}

\begin{proof}
The programs supplied as Online Resource~2 construct all equations from \Cref{eq:master-Hermite} and verify these four statements using exact integer and finite arithmetic.  The runs for the two primes are separate, and each terminates with \texttt{CERTIFICATE VERIFIED}.  The computational supplement describes every step and records the factor degrees, matrix checks, expected output, software versions, and running times.
\end{proof}

\begin{proof}[Proof of \Cref{thm:main}]
It suffices to prove the normalized statement.  Suppose an invariant coupling had a Gaussian coordinate marginal.  Sections~\ref{sec:stationary}--\ref{sec:fourth} would force $F(w,\rho)=0$, and \Cref{sec:sixth} would force $G(w,\rho)=0$.  Hence $w\ne0$, $R(w)=0$, and $\rho=-B_*(w)/A_*(w)$.

Let $T$ be the primitive minimal polynomial of the real algebraic number $w$.  Then $T$ divides $R$.  Use the primary prime $p=23$.  Since the leading coefficient of $R$ is nonzero modulo $23$, the same is true for $T$, and each factor of the reduction of $T$ occurs among the seven factor branches tested in Proposition \ref{prop:finite-certificate}.

After substituting \Cref{eq:rho-branch}, let $d_6(w)$ and $d_8(w)$ be the cleared determinant polynomials for the first five sixth-order rows and the first seven eighth-order rows.  If either determinant vanished at the candidate $w$, Lemma \ref{lem:lifting} would force its reduction to vanish on the corresponding factor branch.  This contradicts the exact matrix checks.  Thus the sixth- and eighth-order unknowns are uniquely determined.

Let $N_8(w)$ be the cleared numerator of the residual in the final eighth-order equation.  Stationarity would require $N_8(w)=0$.  Lemma~\ref{lem:lifting} would then force the reduced residual to be zero on the corresponding factor branch, while Proposition \ref{prop:finite-certificate}(iv) finds it nonzero on every branch.  This contradiction proves that the normalized marginal is not Gaussian.  The inverse affine transformation \Cref{eq:affine-normalization} gives the result. \qedhere
\end{proof}

We end this section with a simple, still useful consequence of the main Theorem.  

The following elementary compactness argument is useful because convergence of $X_t$ alone does not automatically imply convergence of $(X_t,X_{t-1})$.  In addition, the following Corollary gives something more:  under convergence of the law of $X_t$ we have that for large $t$, $X_t$ can not be Gaussian.  

\begin{corollary}[Invariant coupling from marginal convergence]\label{prop:stationary-coupling}
If $\law(X_t)\Rightarrow\mu$, then there is an invariant law $\nu$ for \Cref{eq:state-transition} whose two coordinate marginals are both $\mu$ which is not Gaussian.  In addition, for all large $t$, $X_t$ can not be Gaussian.  
\end{corollary}

\begin{proof} Since this is not the main purpose, we will be brief with this proof and will leave out most of the details.  

Let $\nu_t=\law(X_t,X_{t-1})$ and let $\mathcal P$ be the transition kernel.  Both coordinate marginals of $\nu_t$ converge to $\mu$, so $(\nu_t)$ is tight.  The Ces\`aro averages
\[
 \overline\nu_N=\frac1N\sum_{t=1}^N\nu_t
\]
have weakly convergent subsequences.  The kernel is Feller because the update is continuous and the noise is bounded, and
\[
 \overline\nu_N\mathcal P-\overline\nu_N
 =\frac{\nu_{N+1}-\nu_1}{N}.
\]
Testing against bounded continuous functions and passing to a subsequential limit proves invariance.  Ces\`aro averaging preserves each marginal limit, so both marginals of the limit are $\mu$.

To pass to the conclusion about $X_t$, we just have to notice that because the convergence in distribution is metrizable, we can choose any metric $d$ which does that and argue that $d(\tilde{\mu},N(0,1))>0$ where $\tilde{\mu}$ is the limiting distribution of $\tilde{X}_t :=\frac{X_t}{\sqrt{\mathrm{Var}{X_t}}}$.  Since $d(\tilde{X_t},\tilde{\mu})$ converges to $0$ combined with the main Theorem yields that for large $t$ $d(\tilde{X}_t,N(0,1))>0$, proving the claim.  Notice that we use here the fact that $\mathrm{Var}(X_t)$ indeed converges.  \qedhere  
\end{proof}

\section{The raw-moment formulation and the exact bridge}\label{sec:raw-bridge}

The preceding Hermite proof is the shorter certificate.  A separate raw-moment implementation provides an important independent audit and produces the implicit curves used in the figures.

For the normalized stationary pair, put
\[
 m_{r,s}=\E[X^rY^s],
 \qquad
 \vct M_n=(m_{n,0},m_{n-1,1},\ldots,m_{0,n})^{\mathsf T}.
\]
Write $a=1+w-\phi_1-\phi_2$, $\eta=\phi_2-\phi_1$, and
\[
 \alpha_{i,k}(w,c)=\E[a^i\eta^k]\in\Q[w,c].
\]
Stationarity and the multinomial theorem give, for $r+s=n$,
\begin{equation}\label{eq:raw-master}
 m_{r,s}
 =\sum_{i+j+k=r}\frac{r!}{i!j!k!}(-w)^j
 \alpha_{i,k}\,m_{i+s,j}.
\end{equation}
The $k=0$ terms form a homogeneous operator $T_n$, while $k\ge2$ involves lower even degrees.  Hence
\begin{equation}\label{eq:raw-linear}
 (I-T_n)\vct M_n
 =b_n(w,c;\vct M_0,\vct M_2,\ldots,\vct M_{n-2}).
\end{equation}
At degree $2k$, define the marginal Gaussian defect
\begin{equation}\label{eq:raw-defect}
 \Delta_{2k}=m_{2k,0}-(2k-1)!!\,m_{2,0}^k.
\end{equation}
Fraction-free solution and clearing of common factors yield primitive compatibility polynomials
\[
 Q_4(w,c),\qquad Q_6(w,c),\qquad Q_8(w,c).
\]
The raw implementation uses adjugate identities
\[
 \det(I-T_n)\vct M_n=\operatorname{adj}(I-T_n)b_n,
\]
so its necessary polynomial conditions remain valid even before possible singular moment operators are excluded.

The two derivations meet exactly at fourth order.

\begin{proposition}[Exact bridge between coordinates and eliminants]\label{prop:bridge}
With the normalizations used in the distributed exact files,
\begin{equation}\label{eq:Q4-F-bridge}
 Q_4\bigl(w,(1+w)(1-\rho)\bigr)=-(1+w)^4F(w,\rho).
\end{equation}
Moreover, the primitive degree-$107$ factor of $\Res_c(Q_4,Q_6)$ is exactly the polynomial $R(w)$ in \Cref{eq:resultant-factorization}.
\end{proposition}

\begin{proof}
Both claims are exact coefficient comparisons in $\Z[w,\rho]$ and $\Z[w]$, respectively.  The short script \texttt{code/bridge\_raw\_hermite.py} performs the substitutions, factors the raw resultant, normalizes primitive signs, and asserts equality.  Its transcript is included in the audit package.
\end{proof}

The independent raw-moment certificate also computes
\[
 R_{46}=\Res_c(Q_4,Q_6),
 \qquad
 R_{48}=\Res_c(Q_4,Q_8),
\]
with degrees $208$ and $458$, and proves
\begin{equation}\label{eq:raw-no-triple}
 Q_4(w,c)=Q_6(w,c)=Q_8(w,c)=0
 \quad\Longrightarrow\quad
 \text{no solution with }-1<w<1,\ c>0.
\end{equation}
It separately excludes the singular branches needed for sequential raw-moment inversion:
\begin{align}
&Q_4=0,\quad\det(I-T_4)=0,
&&-1<w<1,\ c>0,\label{eq:sing4}\\
&Q_4=Q_6=0,\quad\det(I-T_6)=0,
&&-1<w<1,\ c>0.\label{eq:sing6}
\end{align}
The checks use exact characteristic-zero resultants, Sturm root isolation, and algebraic sign determination.  Thus the raw and Hermite implementations differ in basis, elimination strategy, and arithmetic backend, yet arrive at the same obstruction.

\section{Geometry of the fourth-, sixth-, and eighth-order conditions}\label{sec:geometry}

The raw-moment calculation produces exact integer polynomials $Q_4(w,c)$, $Q_6(w,c)$, and $Q_8(w,c)$.  Their zero sets are necessary conditions for matching, respectively, the Gaussian fourth, sixth, and eighth moments.  Figure~\ref{fig:curves} uses $c$ on the horizontal axis and $w$ on the vertical axis and displays only the mean-square region \Cref{eq:C2}, whose boundary is
\[
 c=\frac{12(1-w^2)}{7-5w}.
\]

For numerical stability the plotting program writes
\[
 Q_{2k}(w,c)=\sum_j a_j(w)c^j
\]
and evaluates the scaled quantity
\begin{equation}\label{eq:plot-scaling}
 \widehat Q_{2k}(w,c)=
 \frac{Q_{2k}(w,c)}{\sum_j|a_j(w)|c^j}
\end{equation}
by Horner's rule.  For $c\ge0$ the denominator is positive wherever the coefficient vector is nonzero, so \Cref{eq:plot-scaling} has exactly the same zero set as $Q_{2k}$.  A level-zero marching-squares routine is applied on a $1700\times1700$ grid, after masking points outside \Cref{eq:C2}.  The computational supplement reports a comparison with an $850\times850$ grid, direct residual evaluations, and high-precision refinement of the exceptional intersection.  The typical grid-to-grid displacement, measured by the $95\%$ nearest-distance quantile, is between $1.1\times10^{-3}$ and $1.9\times10^{-3}$ in the $(c,w)$ plane.  These figures are explanatory only; no plotted point is used in the proof.

\begin{figure}[htbp]
\centering
\includegraphics[width=0.90\textwidth]{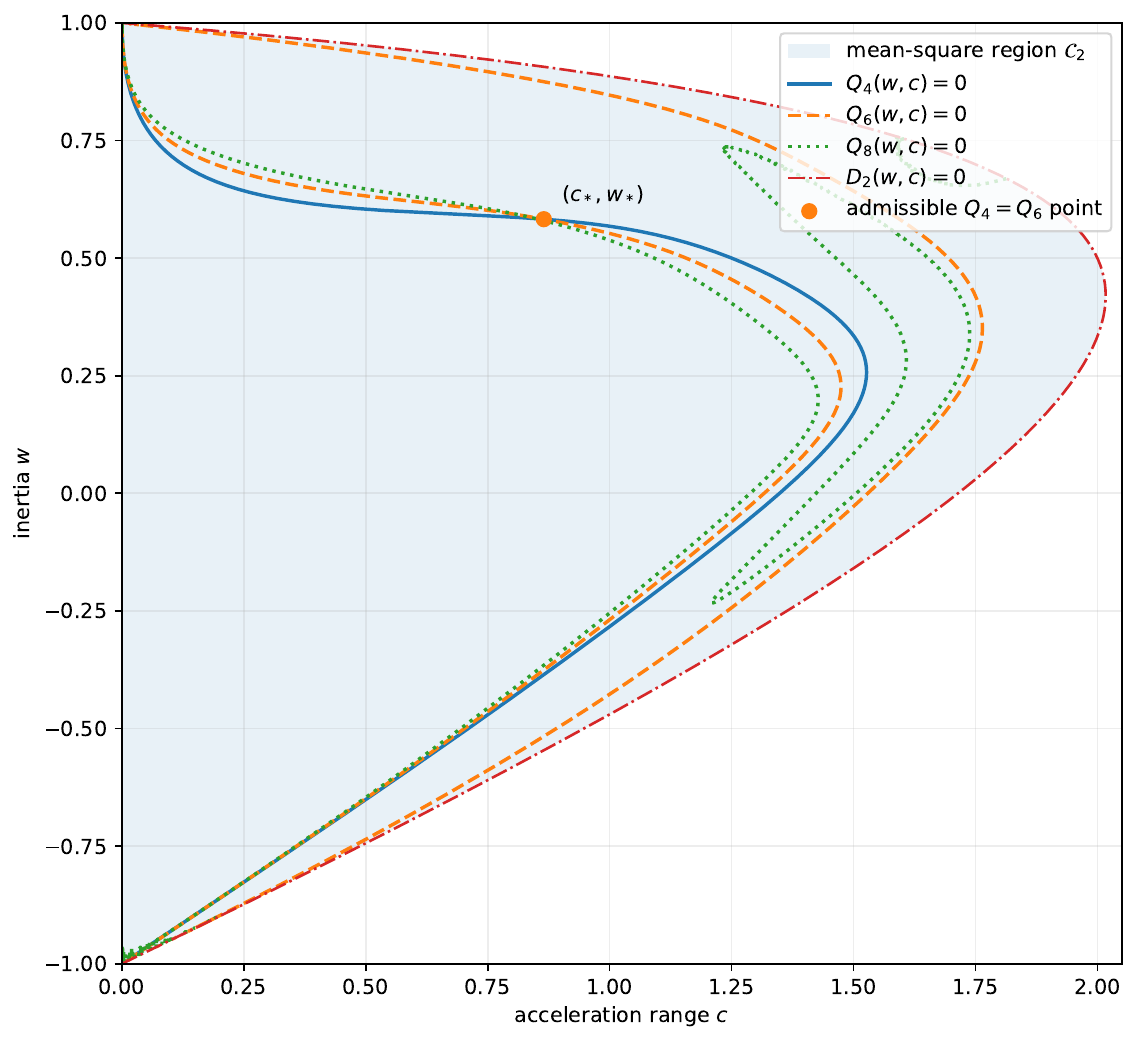}
\caption{Zero sets of the exact moment-defect polynomials $Q_4,Q_6,Q_8$ in the open mean-square region.  The horizontal coordinate is $c$ and the vertical coordinate is $w\in(-1,1)$.  The dash-dotted curve is the mean-square boundary.  The marked point is the admissible $Q_4$--$Q_6$ intersection.}
\label{fig:curves}
\end{figure}

The fourth- and sixth-order conditions genuinely intersect inside the region.  High-precision solution of the two exact polynomial equations gives
\begin{align}\label{eq:exceptional-point}
 w_*&=0.5827476217876752277761786590065\ldots,\notag\\
 c_*&=0.8650336137460312613574937997118\ldots,
\end{align}
with
\begin{equation}\label{eq:rho-star}
 \rho_*=1-\frac{c_*}{1+w_*}
 =0.4534608033282042285057224798475\ldots.
\end{equation}
At this point the standardized stationary moments satisfy
\begin{equation}\label{eq:diagnostic-moments}
 \frac{m_{4,0}}{m_{2,0}^2}=3,
 \qquad
 \frac{m_{6,0}}{m_{2,0}^3}=15,
 \qquad
 \frac{m_{8,0}}{m_{2,0}^4}=106.579793218089576\ldots\ne105.
\end{equation}
Thus fourth and sixth moments alone give a genuine false positive for Gaussianity. Moreover, the excess eighth standardized moment indicates a substantially greater contribution of extreme deviations than in the Gaussian case, providing rigorous support for the tail-fattening phenomenon previously observed empirically in the PSO literature (see, e.g., \citet{blackwell2008examination}).

\begin{figure}[htbp]
\centering
\includegraphics[width=0.90\textwidth]{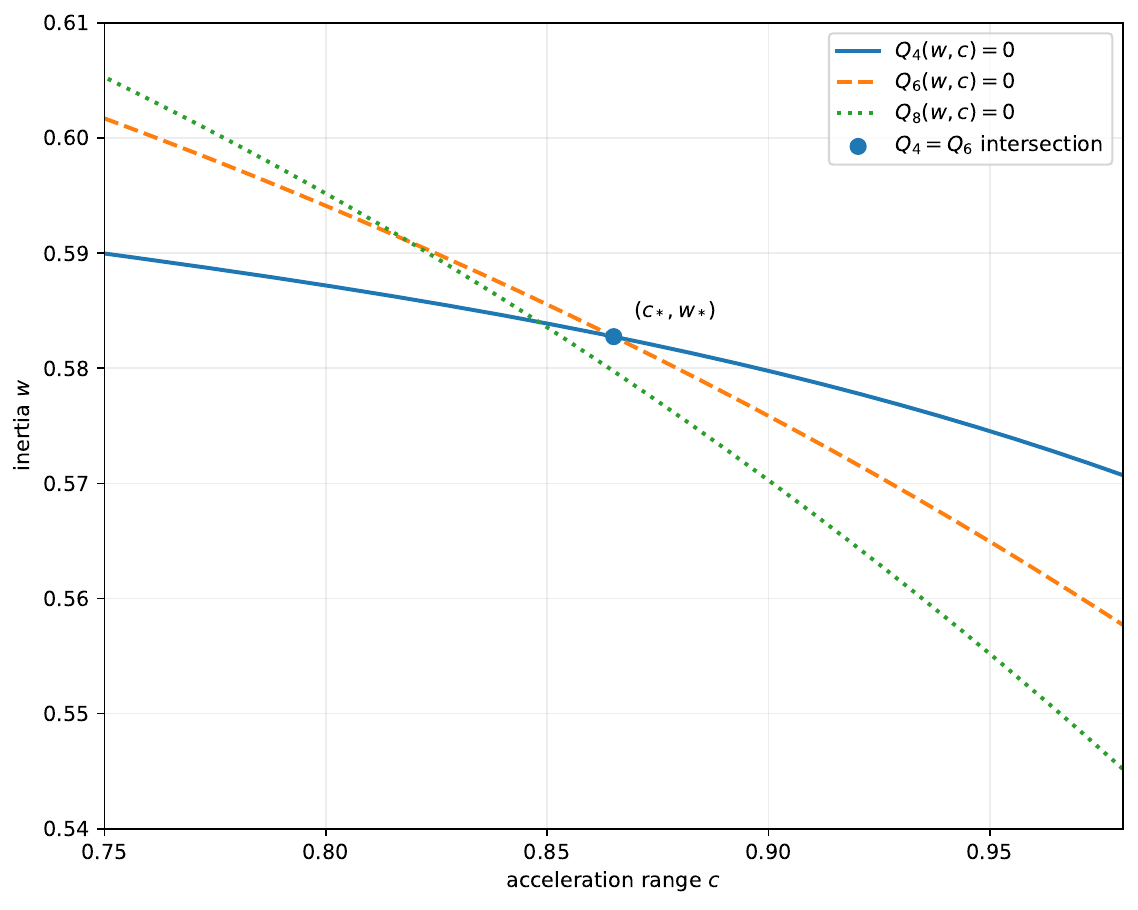}
\caption{Zoom around $(c_*,w_*)$.  The curves $Q_4=0$ and $Q_6=0$ meet, while $Q_8=0$ misses the point.  The nearest point on the plotted $Q_8$ curve is about $2.94\times10^{-3}$ away in the parameter plane.}
\label{fig:zoom}
\end{figure}

This example cautions against diagnosing Gaussianity from kurtosis and one additional even moment.  A random-coefficient autoregressive law can imitate a Gaussian through a surprisingly long finite segment of its moment sequence.

\section{Verification and reproducibility}\label{sec:audit}

The code and the supplementary material can be found at 

\url{https://github.com/ioionel/PSO-and-Gaussian-Conjecture.git}.

The proof has four complementary verification layers, summarized in \Cref{tab:audit}.  The two modular runs rebuild the Hermite equations independently for different primes.  The raw-moment run works over the rational numbers and uses a different basis and elimination path.  The bridge comparison then checks that the two formulations produce exactly the same fourth-order obstruction and the same degree-$107$ eliminant.

\begin{table}[htbp]
\centering
\caption{Verification layers.  ``Exact'' means that no floating-point arithmetic enters the stated check.}
\label{tab:audit}
\begin{tabularx}{\textwidth}{@{}>{\raggedright\arraybackslash}p{0.21\textwidth}>{\raggedright\arraybackslash}X>{\centering\arraybackslash}p{0.12\textwidth}@{}}
\toprule
Layer & What is checked & Result\\
\midrule
Hermite certificate, $p=23$ & Rebuilds the degree-$4,6,8$ systems; derives $F,G,R,A_*,B_*$; checks seven factor branches with degrees $1,2,3,6,10,30,55$. & Exact pass\\
Hermite cross-check, $p=1{,}000{,}003$ & Repeats the same construction with a different factorization, of degrees $1,3,9,11,25,58$. & Exact pass\\
Raw-moment certificate & Checks $Q_4,Q_6,Q_8$, characteristic-zero resultants, real-root isolation, algebraic signs, and singular-system exclusions. & Exact pass\\
Raw/Hermite bridge & Verifies \Cref{eq:Q4-F-bridge} and equality of the two primitive degree-$107$ eliminants coefficient by coefficient. & Exact pass\\
Curve audit & Compares two plotting grids, evaluates scaled polynomial residuals, and refines the $Q_4$--$Q_6$ intersection at 90-digit precision. & Diagnostic pass\\
\bottomrule
\end{tabularx}
\end{table}

All executable material is supplied separately as Online Resource~2.  The computational supplement gives a file-by-file map, exact commands, expected terminal endings, software versions, measured resource use, and a detailed explanation of the curve calculations.  The main theorem does not depend on file hashes or on floating-point plots.


\section{Discussion and extensions}

The moment architecture extends mechanically.  At raw total degree $n$, one builds an $(n+1)\times(n+1)$ system.  In the Hermite basis, Gaussian marginal conditions remove the pure-coordinate unknowns and often reduce the system further.  Orders ten and twelve introduce no conceptual novelty, although coefficient growth becomes severe.

Several extensions are natural.
\begin{itemize}[leftmargin=1.5em]
\item If $\phi_1\sim\operatorname{Unif}[0,c_1]$ and $\phi_2\sim\operatorname{Unif}[0,c_2]$, swap symmetry disappears unless $c_1=c_2$.  Odd total moments then become informative, but the parameter space grows to $(w,c_1,c_2)$.
\item For nonuniform acceleration laws, only the noise moments $\nu_{ij}$ or $\alpha_{i,k}$ change.  The same hierarchy works whenever sufficiently many innovation moments exist.
\item With moving attractors, the fixed-attractor reduction no longer closes.  One must enlarge the state or work conditionally over a stagnation epoch.
\item A characteristic-function proof may exist.  The moment certificate can be interpreted as extracting Taylor coefficients from the exact two-variable characteristic-function equation until incompatibility appears.
\end{itemize}

\section{Conclusion}

For one-dimensional stagnant PSO with equal independent uniform acceleration ranges and two distinct fixed attractors, the stationary position law is not Gaussian anywhere in the open mean-square region.  The fourth and sixth Gaussian moment identities possess a genuine admissible common match, so low-order moment diagnostics alone are insufficient.  The eighth-order stationary equations remove the remaining degree-$107$ eliminant.

The argument is robust in two senses.  First, the Hermite formulation reduces the final large calculation to exact arithmetic modulo a small prime, and the result is repeated with a much larger prime.  Second, an independently organized raw-moment computation reproduces the same fourth-order curve and the same degree-$107$ eliminant, then reaches the same conclusion over the rational numbers.  The empirical bell curve is therefore Gaussian-like, but not Gaussian.

\section*{Statements and Declarations}

\noindent\textbf{Funding.} Ionel Popescu was supported by Informational Buildup Foundation (IBF) through the IBF Research Fellowships granted via Institute of Mathematics of the Romanian Academy. 

\noindent\textbf{Competing interests.} The authors declare that they have no known competing financial interests or personal relationships that could have appeared to influence the work reported in this paper.


\noindent\textbf{Data availability.} No empirical data set is used.  The exact polynomial artifacts and numerical validation tables are supplied at \url{https://github.com/ioionel/PSO-and-Gaussian-Conjecture}.

\noindent\textbf{Code availability.} The complete verification code, plotting code, exact inputs, and expected logs are supplied at \url{https://github.com/ioionel/PSO-and-Gaussian-Conjecture}.

\noindent\textbf{Use of generative AI.} During preparation of the manuscript, ChatGPT was used to assist with symbolic-computation scripting, manuscript organization, and language revision.  Every computer-generated formula and every computational certificate reported in the paper was independently reviewed and checked by exact executable tests.  The authors provided the ideas, the methods and take full responsibility for the content.

\noindent\textbf{Corresponding author email address:} \url{ionel.popescu@fmi.unibuc.ro}.


\bibliographystyle{plainnat}
\bibliography{references}
\end{document}